\documentclass[leqno]{amsart}
\usepackage{amsmath,amsthm,amsfonts,amssymb}

\numberwithin{equation}{section}
\theoremstyle{plain}

\newtheorem{theorem}{Theorem}

\newtheorem{definition}[theorem]{Definition}

\newtheorem{lemma}[theorem]{Lemma}

\newtheorem{remark}[theorem]{Remark}

\begin{document}

\title[Limit Theorems for Empirical Density of Greatest Common Divisors]
{Limit Theorems for Empirical Density of Greatest Common Divisors}

\author{Behzad Mehrdad}

\author{Lingjiong Zhu}
\address
{Courant Institute of Mathematical Sciences\newline
\indent New York University\newline
\indent 251 Mercer Street\newline
\indent New York, NY-10012\newline
\indent United States of America}
\email{mehrdad@cims.nyu.edu
\\
ling@cims.nyu.edu}

\date{28 May 2015. \textit{Revised:} 18 April 2016}

%large deviations, Central limit and other weak theorems, Multiplicative structure; Euclidean algorithm; greatest common divisors,

\begin{abstract}
The law of large numbers for the empirical density for the pairs of uniformly distributed integers
with a given greatest common divisor is a classic result in number theory. In this paper,
we study the large deviations of the empirical density. We will also obtain a rate of convergence
to the normal distribution for the central limit theorem.
Some generalizations are provided.
\end{abstract}

\maketitle

\section{Introduction}

Let $X_{1},\ldots,X_{n}$ be the random variables uniformly distributed on $\{1,2\ldots,n\}$.
It is well known that
\begin{equation}\label{LLN_1}
\frac{1}{n^{2}}\sum_{1\leq i,j\leq n}1_{\text{gcd}(X_{i},X_{j})=\ell}\rightarrow\frac{6}{\pi^{2}\ell^{2}},
\quad\ell\in\mathbb{N},
\end{equation}
in probability as $n\rightarrow\infty$.

The intuition is the following. If the law of large numbers holds, the limit is $\mathbb{P}(\text{gcd}(X_{1},X_{2})=\ell)$.  Let $X_{1},X_{2}\in C_{\ell}:=\{\ell n:n\in\mathbb{N}\}$ that happens with probability $\frac{1}{\ell^{2}}$ as $n\rightarrow\infty$. Observe that
\begin{equation}
\{\text{gcd}(X_{1},X_{2})=\ell\}=\{X_{1},X_{2}\in C_{\ell},\text{gcd}(X_{1}/\ell,X_{2}/\ell)=1\},
\end{equation}
where $\{X_{1},X_{2}\in C_{\ell}\}$ and $\{\text{gcd}(X_{1}/\ell,X_{2}/\ell)=1\}$
are asymptotically independent.  Therefore, we get \eqref{LLN_1} by noticing that $\sum_{\ell=1}^{\infty}\frac{1}{\ell^{2}}=\frac{\pi^{2}}{6}$.

On the other hand, two independent uniformly chosen integers are coprime if and only if they do not have a common prime factor.
For any prime number $p$, the probability that a uniformly random integer is divisible by $p$ is $\frac{1}{p}$ as $n$ goes to infinity.
Hence, we get an alternative formula,
\begin{equation}\label{LLN}
\frac{1}{n^{2}}\sum_{1\leq i,j\leq n}1_{\text{gcd}(X_{i},X_{j})=1}\rightarrow\prod_{p\in\mathcal{P}}\left(1-\frac{1}{p^{2}}\right)
=\frac{1}{\zeta(2)}=\frac{6}{\pi^{2}},
\end{equation}
in probability as $n\rightarrow\infty$,
where $\zeta(\cdot)$ is the Riemann zeta function and throughout this paper $\mathcal{P}$ denotes the set
of all the prime numbers in an increasing order.

The fact that $\mathbb{P}(\text{gcd}(X_{i},X_{j})=1)\rightarrow\frac{6}{\pi^{2}}$ was first proved by Ces\`{a}ro \cite{CesaroIII}.
The identity relating the product over primes to $\zeta(2)$ in \eqref{LLN}
is an example of an Euler product, and the evaluation of $\zeta(2)$ as 
$\pi^{2}/6$ is the Basel problem, solved by Leonhard Euler in 1735.
For a discussion on the probability of two numbers being coprime is $\frac{6}{\pi^{2}}$, we also refer to 
Theorem 332 in Hardy and Wright \cite{Hardy}. 
For further details and properties of the distributions, moments and asymptotic for the greatest common divisors,
we refer to Ces\`{a}ro \cite{CesaroI}, \cite{CesaroII}, Cohen \cite{Cohen}, Diaconis and Erd\H{o}s \cite{Diaconis}
and Fern\'{a}ndez and Fern\'{a}ndez \cite{FernandezI}, \cite{FernandezII}.

Since the law of large numbers result is well-known, it is natural to study the fluctuations, i.e. central limit theorem
and the probabilities of rare events, i.e. large deviations. The central limit theorem was recently
obtained in Fern\'{a}ndez and Fern\'{a}ndez \cite{FernandezI} and the possibility of the rate of convergence
to normality was also mentioned there. 
We will provide the rate of convergence to normal distribution. 
The main contribution of our paper is the large deviations result 
and the proofs are considerably more involved. 

For the readers who are interested in the probabilistic methods in number theory, we refer to the books
by Elliott \cite{Elliott} and Tenenbaum \cite{Tenenbaum}.

The paper is organized in the following way. In Section \ref{MainSection}, we state the main results, i.e.
the central limit theorem and the convergence rate to the Gaussian distribution and the large deviation principle
for the empirical density. 
The proofs for large deviation principle are given in Section \ref{LDPProofs},
and the proofs for the central limit theorem are given in Section \ref{CLTProofs}.

\section{Main Results}\label{MainSection}

\subsection{Central Limit Theorem}

In this section, we will show a central limit theorem and obtain the  rate of convergence
to the normal distribution. The method we will use is based on a Stein's method result 
for central limit theorems, that is, Theorem 3.6. in Ross \cite{Ross}. 
Before we proceed, for any two random variables $W,Z$, let us define the Wasserstein distance $d_{W}$ as
\begin{equation}
d_{\mathcal{H}}(W,Z)=\sup_{h\in\mathcal{H}}
\left|\int h(x)d\mu(x)-\int h(x)d\nu(x)\right|,
\end{equation}
where $\mu$ and $\nu$ are the probability laws of $W$ and $Z$, and
\begin{equation}
\mathcal{H}:=\left\{h:\mathbb{R}\rightarrow\mathbb{R}, |h(x)-h(y)|\leq|x-y|\right\}.
\end{equation}

\begin{theorem}\label{CLTThm}
Let $Z$ be a standard normal distribution with mean $0$ and variance $1$, and let $\ell\in\mathbb{N}$ .
\begin{equation}
d_{W}\left(\frac{\sum_{1\leq i,j\leq n}1_{\text{gcd}(X_{i},X_{j})=\ell}-n^{2}\frac{6}{\ell^{2}\pi^{2}}}
{2\sigma n^{3/2}},Z\right)\leq\frac{C_{\ell}}{n^{1/2}},
\end{equation}
where $C_{\ell}>0$ is a constant that depends only on $\ell$ such that $C_{\ell}=O(\ell^{5/2})$
as $\ell\rightarrow\infty$ and
\begin{equation}
\sigma^{2}:=\frac{1}{\ell^{3}}\prod_{p\in\mathcal{P}}
\left(1-\frac{2}{p^{2}}+\frac{1}{p^{3}}\right)-\frac{36}{\ell^{4}\pi^{4}}.
\end{equation}
\end{theorem}

\subsection{Large Deviation Principle}

In this section, we are interested to study the following probability,
\begin{equation}
\mathbb{P}\left(\frac{1}{n^{2}}\sum_{1\leq i,j\leq n}1_{\text{gcd}(X_{i},X_{j})=\ell}\in A\right)
,\qquad\text{as $n\rightarrow\infty$},
\end{equation}
Indeed, later we will see that,
\begin{equation}
\mathbb{P}\left(\frac{1}{n^{2}}\sum_{1\leq i,j\leq n}1_{\text{gcd}(X_{i},X_{j})=\ell}\in A\right)
=e^{-nH(A)+o(n)},
\quad\ell\in\mathbb{N},
\end{equation}
where $H(A)=0$ if $\frac{1}{\ell^{2}}\frac{6}{\pi^{2}}\in A$ 
and $H(A)>0$ if $\frac{1}{\ell^{2}}\frac{6}{\pi^{2}}\notin A$, i.e. this probability decays exponentially 
fast as $n\rightarrow\infty$ if the empirical mean deviates aways from the ergodic mean. 
This phenomenon is called large deviations in probability theory. 

Before we proceed, let us introduce the formal definition of large deviations.  
A sequence $(P_{n})_{n\in\mathbb{N}}$ of probability measures on a topological space $X$ 
satisfies the large deviation principle with rate function $I:X\rightarrow\mathbb{R}$ if $I$ is non-negative, 
lower semicontinuous and for any measurable set $A$, 
\begin{equation}
-\inf_{x\in A^{o}}I(x)\leq\liminf_{n\rightarrow\infty}\frac{1}{n}\log P_{n}(A)
\leq\limsup_{n\rightarrow\infty}\frac{1}{n}\log P_{n}(A)\leq-\inf_{x\in\overline{A}}I(x).
\end{equation}
Here, $A^{o}$ is the interior of $A$ and $\overline{A}$ is its closure. 
We refer to Dembo and Zeitouni \cite{Dembo} or Varadhan \cite{VaradhanII} for general background of large deviations and the applications.

For the moment, let us concentrate on the case in which we consider
the number of coprime pairs. Before we state our main result, let us introduce some notations and definitions first.
Let $S:=(s_{i})_{i\in\mathbb{N}}$ be a sequence of numbers on $[0,1]$. We define
the probability measure $\nu_{k}^{S}$ on $[0,1]$, for $k\in\mathbb{N}$, as follows.
\begin{equation}
\nu_{k}^{S}([0,b]):=\sum_{i=1}^{k}\prod_{j=1}^{i-1}(1-s_{j})(1-s_{i})^{b_{i}},
\end{equation}
where $b=0.b_{1}b_{2}\ldots$ is the binary expansion of $b$. As for the binary expansion of $b$, we always take the finite expansion, whenever there are more than one  representation. However, that does not have any effect on our problem, since the set of such numbers is countable and has measure zero under $\nu_{k}^{S}$, for any $k\in\mathbb{N}$.

Now, if we draw a random variable $U^{k}$
according to the measure $\nu_{k}^{S}$ and consider the first $k$ digits in the binary expansion of $U^{k}$, 
they are distributed as $k$ Bernoulli random variables with parameters $(s_{i})_{i=1}^{k}$. It is easy
to see that a measure $\nu^{S}$ exists as a weak limit of $\nu_{k}^{S}$ and let $\nu^{S}$ be its weak limit.
For example, if $s_{i}=\frac{1}{2}$, for $i\in\mathbb{N}$, then $\nu^{S}$ is simply the Lebesgue measure on $[0,1]$.
Let $(p_{i})_{i\in\mathbb{N}}$ be the members of $\mathcal {P}$ in the increasing order. From now on,
we work with $\nu_{k}$ and $\nu$, for which the $s_{i}$ is $\frac{1}{p_{i}}$, or 
\begin{equation}\label{my_nu}
\nu=\nu^{P}, \text{ where } P:=\left( \frac{1}{p_{i}}\right)_{i\in\mathbb{N}}.
\end{equation}

In addition, for $a\in[0,1]$ and $i\in\mathbb{N}$, we define
\begin{equation}\label{chi}
\chi_{i}(a)=\text{the $i$th digit in the binary expansion of $a$}.
\end{equation}
We also define $f:[0,1]^{2}\rightarrow\{0,1\}$, for $k\in\mathbb{N}$, as follows
\begin{equation}\label{my_rate}
f(x_{1},x_{2}):=1-\max_{i\in\mathbb{N}}\chi_{i}(x_{1})\chi_{i}(x_{2}).
\end{equation}
In other words, $f(x_{1},x_{2})$ is $1$ if $x_{1}$ and $x_{2}$ do not share a common $1$ at the same place in their binary expansions
and $f$ is $0$ otherwise. Now, we are ready to state our main result.

\begin{theorem}\label{LDPThm}
Recall that random variables $X_{i},\ldots,X_{n}$, are distributed uniformly on $\{1,2\ldots,n\}$. The probability measures
$\mathbb{P}\left(\frac{1}{n^{2}}\sum_{1\leq i,j\leq n}1_{\text{gcd}(X_{i},X_{j})=1}\in\cdot\right)$ satisfy
a large deviation principle with rate function
\begin{equation}\label{I_1}
I_{1}(x)=\inf_{\iint_{[0,1]^{2}}f(x_{1},x_{2})\mu(dx_{1})\mu(dx_{2})=x}\int_{[0,1]}\log\left(\frac{d\mu}{d\nu}\right)d\mu,
\end{equation}
where $\nu$ and $f$ are defined in \eqref{my_nu} and \eqref{my_rate}, respectively.
\end{theorem}

Let us get some intuition with \eqref{I_1}, before we see our next result. For $X\in\mathbb{N}$ and $p\in\mathcal{P}$,  the indicator $\textbf{1}_{p|X}$ is $1$ if $p$ divides $X$, and $0$ otherwise.  We let $a\in[0,1]$ be a number such that $\chi_{i}(a)=\textbf{1}_{p_{i}|X}$, where $p_{i}$  is the $i$th prime in $\mathcal{P}$ and $i\in\mathbb{N}$. In other words, the $i$th digit in the binary expansion of $a$ shows whether $X$  is divisible by $p_{i}$ or not.  We also define 
\begin{equation}
\psi : \mathbb{N}\to [0,1] \quad \text{as} \quad \psi(X):=a.
\end{equation}

Now, for integers $X,Y\in\mathbb{N}$, $gcd(X,Y)$ is $1$ if and only if, for every $p\in\mathcal{P}$, $p$ does not divide both $X$ and $Y$. So, comparing this with the definition \eqref{my_rate} of $f$, we get 
\begin{equation}
f(\psi (X),\psi(Y))=\textbf{1}_{gcd(X,Y)=1}.
\end{equation}

Therefore, our problem is to show large deviation principle for probability measures 
\begin{equation}
\mathbb{P}\left(\frac{1}{n^{2}}\sum_{1\leq i,j\leq n}{f(\psi(X_{i}), \psi(X_{j}))}\in\cdot\right),
\end{equation}
where $X_{i}$, for $1\leq i\leq n$, are distributed uniformly on $\{1,2\ldots,n\}$. We note that, for $p,q\in\mathcal{P}$ and as $n$ goes to infinity, the probabilities for the events $\{p|X_{1}\}$, $\{q|X_{1}\}$ and $\{pq|X_{1}\}$ approach to $\frac{1}{p}$, $\frac{1}{q}$ and $\frac{1}{pq}$, respectively. Hence, as $n$ goes to infinity, the underlying measure of $\psi(X_{1})$ looks more like $\nu$. Although this is not precise, for large $n$, $\psi(X_{1})\cdots \psi(X_{n})$ are $n$ i.i.d. random variables with measure $\nu$. Thus, our hope is to use Sanov's theorem to obtain large deviation principle for random variables $\psi(X_{i})$, and then, we use the 
contraction principle with the map $f$ to get the rate function \eqref{I_1}. 
 
There are a few issues on our way that need to be addressed, e.g. $\psi(X_{i})$ , for $1\leq i\leq n$, are not distributed as $\nu$
and the mapping $f$ is not continuous at any point (to apply the contraction principle, the mapping is usually assumed to be
continuous). We will come back to these obstacles in the proof section along with the statement of Sanov's theorem and the contraction principle.

We can also consider the following large deviation problem,
\begin{equation}
\mathbb{P}\left(\frac{1}{n^{2}}\sum_{1\leq i,j\leq n}1_{\text{gcd}(X_{i},X_{j})=\ell}\in\cdot\right).
\end{equation}
Write 
\begin{equation}
\ell=q_{1}^{\beta_{1}}q_{2}^{\beta_{2}}\cdots q_{m}^{\beta_{m}},
\end{equation}
where $q_{i}$ are distinct primes and $\beta_{i}$ are positive integers for $1\leq i\leq m$. 

For a fixed $\ell$, let $p_{1},\ldots,p_{k}$ be the smallest
$k$ primes distinct from $q_{1},\ldots,q_{m}$. Any positive integer can be written as
\begin{equation}
q_{1}^{\gamma_{1}}\cdots q_{m}^{\gamma_{m}}p_{1}^{\alpha_{1}}\cdots p_{k}^{\alpha_{k}},
\end{equation}
where $\gamma_{i}$ and $\alpha_{j}$ are non-negative integers.
Any number on $[0,1]$ can be written as
\begin{equation}
0.\gamma_{1}\gamma_{2}\cdots\gamma_{m}\alpha_{1}\alpha_{2}\cdots\alpha_{k}\cdots,
\end{equation}
where $\gamma_{1},\ldots,\gamma_{m}$ are obtained from ternary expansion and $\alpha_{1},\alpha_{2},\ldots$
are obtained from binary expansion. 

The interpretation is that if an integer is not divisible by $q_{i}^{\beta_{i}}$, then $\gamma_{i}=0$.
If it is divisible by $q_{i}^{\beta_{i}}$ but not by $q_{i}^{\beta_{i}+1}$, then $\gamma_{i}=1$.
Finally, if it is divisible by $q_{i}^{\beta_{i}+1}$, then $\gamma_{i}=2$.
We also have $\alpha_{j}=0$ if an integer is not divisible by $p_{j}$ and $1$ otherwise.

Restrict to the first $m+k$ digits and define a probability measure $\nu_{k}$ that takes values
\begin{equation}
g(q_{1})\cdots g(q_{m})\left(\frac{1}{p_{1}}\right)^{\alpha_{1}}\left(1-\frac{1}{p_{i}}\right)^{1-\alpha_{1}}
\cdots\left(\frac{1}{p_{k}}\right)^{\alpha_{k}}\left(1-\frac{1}{p_{k}}\right)^{1-\alpha_{k}},
\end{equation}
where
\begin{equation}
g(q_{i})
=
\begin{cases}
1-\frac{1}{q_{i}^{\beta_{i}}} &\text{if $\gamma_{i}=0$}
\\
\frac{1}{q_{i}^{\beta_{i}}}-\frac{1}{q_{i}^{\beta_{i}+1}} &\text{if $\gamma_{i}=1$}
\\
\frac{1}{q_{i}^{\beta_{i}+1}} &\text{if $\gamma_{i}=2$}
\end{cases},
\quad 1\leq i\leq m.
\end{equation}

Let $\nu$ be the weak limit of $\nu_{k}$. We get the following result. The proofs are similar
to that of Theorem \ref{LDPThm} and are omitted here. 

%Following the same proofs as in Theorem \ref{LDPThm}, 
%%(the superexponential estimates also hold here), 
% get the following result.

\begin{theorem} \label{LDPThm2}
For $\ell>1$, the probability measures
$\mathbb{P}\left(\frac{1}{n^{2}}\sum_{1\leq i,j\leq n}1_{\text{gcd}(X_{i},X_{j})=\ell}\in\cdot\right)$ satisfy
a large deviation principle with rate function
\begin{equation}
I_{\ell}(x)=\inf_{\iint_{[0,1]^{2}}f_{\ell}(x_{1},x_{2})\mu(dx_{1})\mu(dx_{2})=x}\int_{[0,1]}\log\left(\frac{d\mu}{d\nu}\right)d\mu,
\end{equation}
where $f_{\ell}(x_{1},x_{2})=1$ if $x_{1}$ and $x_{2}$ do not share a $2$ in their first $m$ digits or a common
$1$ in the rest of the expansion, and none of them have any $0$ in their first $m$ digits. Otherwise, $f_{\ell}(x_{1},x_{2})=0$.
\end{theorem}

\begin{remark}
It is interesting to observe that $\frac{6}{\pi^{2}}$ is also the density of square-free integers. That is because an
integer is square-free if and only if it is not divisible by $p^{2}$ for any prime number $p$.
Therefore, we have the law of large numbers, i.e.
\begin{equation}
\frac{1}{n}\sum_{i=1}^{n}1_{X_{i}\text{ is square-free}}\rightarrow
\prod_{p\in\mathcal{P}}\left(1-\frac{1}{p^{2}}\right)=\frac{6}{\pi^{2}},.
\end{equation}
in probability as $n\rightarrow\infty$.
The central limit theorem is standard,
\begin{equation}
\frac{\sum_{i=1}^{n}1_{X_{i}\text{ is square-free}}-\frac{6n}{\pi^{2}}}{\sqrt{n}}\rightarrow
N\left(0,\frac{6}{\pi^{2}}-\frac{36}{\pi^{4}}\right),
\end{equation}
in distribution as $n\rightarrow\infty$.
The large deviation principle also holds with rate function
\begin{equation}
I(x):=x\log\left(\frac{x}{6/\pi^{2}}\right)+(1-x)\log\left(\frac{1-x}{1-6/\pi^{2}}\right).
\end{equation}
\end{remark}

\begin{remark}
One can also generalize the result to ask what it is the probability that if we uniformly randomly choose $d$ numbers
from $\{1,2,\ldots,n\}$ their greatest common divisor is $1$.
It is not hard to see that
\begin{equation}
\frac{1}{n^{d}}\sum_{1\leq i_{1},\ldots,i_{d}\leq n}1_{\text{gcd}(X_{i_{1}},\ldots,X_{i_{d}})=1}
\rightarrow\prod_{p\in\mathcal{P}}\left(1-\frac{1}{p^{d}}\right)=\frac{1}{\zeta(d)},
\end{equation}
in probability as $n\rightarrow\infty$.
There are $d^{2}n(n-1)\cdots(n-(2d-2))$ pairs $(i_{1},\ldots,i_{d})$ and $(j_{1},\ldots,j_{d})$ so that 
$|\{i_{1},\ldots,i_{d}\}\cap\{j_{1},\ldots,j_{d}\}|=1$. It is also easy to see that
\begin{equation}
\mathbb{P}\left(\text{gcd}(X_{1},\ldots,X_{d})=\text{gcd}(X_{d},\ldots,X_{2d-1})=1\right)
=\prod_{p\in\mathcal{P}}\left(1-\frac{2}{p^{d}}+\frac{1}{p^{2d-1}}\right).
\end{equation}
Therefore, we have the central limit theorem.
\begin{align}
&\frac{1}{d\cdot n^{\frac{2d-1}{2}}}
\left\{\sum_{1\leq i_{1},\ldots,i_{d}\leq n}1_{\text{gcd}(X_{i_{1}},\ldots,X_{i_{d}})=1}
-n^{d}\prod_{p\in\mathcal{P}}\left(1-\frac{1}{p^{d}}\right)\right\}
\\
&\rightarrow
N\left(0,\prod_{p\in\mathcal{P}}\left(1-\frac{2}{p^{d}}+\frac{1}{p^{2d-1}}\right)
-\prod_{p\in\mathcal{P}}\left(1-\frac{1}{p^{d}}\right)^{2}\right),\nonumber
\end{align}
in distribution as $n\rightarrow\infty$.
We also have that $\mathbb{P}(\frac{1}{n^{d}}\sum_{1\leq i_{1},\ldots,i_{d}\leq n}1_{\text{gcd}(X_{i_{1}},\ldots,X_{i_{d}})=1}\in\cdot)$
satisfies a large deviation principle with the rate function
\begin{equation}
I(x)=\inf_{\idotsint_{[0,1]^{d}}f(x_{1},x_{2},\ldots,x_{d})\mu(dx_{1})\cdots\mu(dx_{d})=x}\int_{[0,1]}\log\left(\frac{d\mu}{d\nu}\right)d\mu,
\end{equation}
where $\nu$ is the same as in Theorem \ref{LDPThm} and
\begin{equation}
f(x_{1},\ldots,x_{d})=
\begin{cases}
1 &\text{if $x_{1},\ldots,x_{d}$ do not share a common $1$ in their binary expansions}
\\
0 &\text{otherwise}
\end{cases}.
\end{equation}
\end{remark}

%================================================================================
%================================================================================
%================================   Proof  Sec    =====================================
%================================================================================
%================================================================================

\section{Proofs of Large Deviation Principle}\label{LDPProofs}

The proof is the discussion that follows Theorem \ref{LDPThm}. In order to make that precise, we need
to prove a series of lemmas and theorems of superexponential estimates. It is also worth mentioning that the proof of Theorem \ref{LDPThm2} is very close to that of Theorem \ref{LDPThm} and we skip it.

Let us give the definitions of $Y_{p}$, $S(k_{1},k_{2})$, $\tilde{X}_{i}$, and $\tilde{Y}_{p}$ that will be used repeatedly throughout
this section.

\begin{definition}
For any prime number $p$, we define
\begin{equation}
Y_{p}:=\#\{1\leq i\leq n: X_{i}\text{ is divisible by $p$}\}.
\end{equation}
\end{definition}

\begin{definition}
For any $k_{1},k_{2}\in\mathbb{N}$, let us define
\begin{equation}
S(k_{1},k_{2}):=\{p\in\mathcal{P}:k_{1}<p\leq k_{2}\}.
\end{equation}
\end{definition}

\begin{definition}
We define i.i.d. $\mathbb{N}$ valued random variables $\tilde{X}_{i}$ such that $\mathbb{P}(\text{$\tilde{X}_{i}$ is divisible
by $p$})=\frac{1}{p}$ for any $p\in\mathcal{P}$, $p\leq n$, 
and $\mathbb{P}(\text{$\tilde{X}_{i}$ is divisible
by $p$})=0$ for any $p\in\mathcal{P}$, $p>n$, and the events $\{\text{$\tilde{X}_{i}$ divisible by $p$}\}$ and 
$\{\text{$\tilde{X}_{i}$ divisible by $q$}\}$ 
are independent for distinct $p,q\in\mathcal{P}$, $p,q\leq n$. 
We define $\tilde{Y}_{p}$ as
$\tilde{Y}_{p}:=\#\{1\leq i\leq n: \tilde{X}_{i}\text{ is divisible by $p$}\}$.
\end{definition}

\begin{lemma}\label{MGF}
Let $\tilde{Y}$ be a Binomial random variable distributed as $B(\alpha,n)$. For any $\lambda>0$, let $\lambda_{1}:=e^{\lambda}$.
If $2\alpha\lambda_{1}^{2}<1$ and $\alpha<\frac{1}{2}$, then, for sufficiently large $n$,
\begin{equation}
\frac{1}{n}\log\mathbb{E}\left[e^{\frac{\lambda}{n}\tilde{Y}^{2}}\right]
\leq 4\lambda\alpha^{2}\lambda_{1}^{4}+\frac{\log 4(n+1)}{n}.
\end{equation} 
\end{lemma}

\begin{proof}
By the definition of Binomial distribution, 
\begin{align}
\mathbb{E}\left[e^{\frac{\lambda}{n}\tilde{Y}^{2}}\right]
&=\sum_{i=0}^{n}\binom{n}{i}\alpha^{i}(1-\alpha)^{n-i}e^{\frac{\lambda i^{2}}{n}}
\\
&\leq(n+1)\max_{0\leq i\leq n}\binom{n}{i}\alpha^{i}(1-\alpha)^{n-i}e^{\frac{\lambda i^{2}}{n}}.\nonumber
\end{align}
Using Stirling's formula, for any $n\in\mathbb{N}$,
\begin{equation}
1\leq\frac{n!}{\sqrt{2\pi n}(n/e)^{n}}\leq\frac{e}{\sqrt{2\pi}}.
\end{equation}
Therefore, we have $\binom{n}{i}\leq 4e^{nH(i/n)}$, where 
$H(x):=-x\log x-(1-x)\log(1-x)$, $0\leq x\leq 1$. Hence, 
\begin{align}
&\frac{1}{n}\log\mathbb{E}\left[e^{\frac{\lambda}{n}\tilde{Y}^{2}}\right]
\\
&\leq\frac{\log 4(n+1)}{n}
+\max_{0\leq i\leq n}\left\{H\left(\frac{i}{n}\right)+\frac{i}{n}\log(\alpha)+\left(1-\frac{i}{n}\right)\log(1-\alpha)
+\lambda\left(\frac{i}{n}\right)^{2}\right\}.\nonumber
\end{align}
To find the maximum of
\begin{equation}
f(x):=H(x)+x\log(\alpha)+(1-x)\log(1-\alpha)+\lambda x^{2},
\end{equation}
it is sufficient to look at
\begin{equation}\label{maxI}
f'(x)=\log\left(\frac{\alpha}{1-\alpha}\right)-\log\left(\frac{x}{1-x}\right)+2\lambda x.
\end{equation}
The assumptions $2\alpha\lambda_{1}^{2}<1$ and $\alpha<\frac{1}{2}$ implies that
\begin{equation}\label{maxII}
\frac{\alpha}{1-\alpha}\lambda_{1}^{2}
\leq 2\alpha\lambda_{1}^{2}\leq\frac{2\alpha\lambda_{1}^{2}}{1-2\alpha\lambda_{1}^{2}}.
\end{equation}
Since logarithm is an increasing function, \eqref{maxI} and \eqref{maxII} imply
that $f'(x)<0$ for any $x\geq 2\alpha\lambda_{1}^{2}$. Therefore, the maximum of $f$ is attained
at some $x\leq 2\alpha\lambda_{1}^{2}$.

In addition, since $\log(\frac{x}{1-x})$ is increasing in $x$, the maximum of
\begin{equation}
g(x):=H(x)+x\log(\alpha)+(1-x)\log(1-\alpha)
\end{equation}
is achieved at $x=\alpha$, which is $g(\alpha)=0$. Hence,
\begin{equation}
\max_{0\leq x\leq 1}f(x)=\max_{0\leq x\leq 2\alpha\lambda_{1}}f(x)\leq 0+\lambda x^{2}\leq\lambda(2\alpha\lambda_{1}^{2})^{2},
\end{equation}
which concludes the proof.
\end{proof}

\begin{theorem}\label{SuperI}
For any $k,n\in\mathbb{N}$ sufficiently large and $\epsilon>0$,
\begin{equation}
\frac{1}{n}\log\mathbb{P}\left(\sum_{p\in S(k,n)}\tilde{Y}_{p}^{2}>n^{2}\epsilon\right)
\leq-\frac{\epsilon}{8}\log(k)+4.
\end{equation}
Therefore, we have the following superexponential estimate,
\begin{equation}
\limsup_{k\rightarrow\infty}\limsup_{n\rightarrow\infty}
\frac{1}{n}\log\mathbb{P}\left(\sum_{p\in S(k,n)}\tilde{Y}_{p}^{2}>n^{2}\epsilon\right)
=-\infty.
\end{equation}
\end{theorem}

\begin{proof}
Note that $\tilde{Y}_{p}=\#\{1\leq i\leq n:\tilde{X}_{i}\text{ is divisible by $p$}\}$. And
whether $\tilde{X}_{i}$ is divisible by $p$ is independent from $\tilde{X}_{i}$ being divisible by $q$
for distinct primes $p$ and $q$. In other words, $\tilde{Y}_{p}$ are independent for distinct primes $p\in\mathcal{P}$.
By Chebyshev's inequality, for any $\lambda>0$,
\begin{align}
\frac{1}{n}\log\mathbb{P}\left(\sum_{p\in S(k,n)}\tilde{Y}_{p}^{2}>n^{2}\epsilon\right)
&\leq-\lambda\epsilon+\frac{1}{n}\log\mathbb{E}\left[e^{\frac{\lambda}{n}\sum_{p\in S(k,n)}\tilde{Y}_{p}^{2}}\right]
\label{UpperZero}
\\
&=-\lambda\epsilon+\frac{1}{n}\sum_{p\in S(k,n)}\log\mathbb{E}\left[e^{\frac{\lambda}{n}\tilde{Y}_{p}^{2}}\right].
\nonumber
\end{align}
We choose $k\in\mathbb{N}$ large enough so that $\lambda_{1}=e^{\lambda}<\sqrt{2k}$. For $k<p\leq n$, we have
$\frac{2}{p}\lambda_{1}^{2}<\frac{2}{k}\lambda_{1}^{2}<1$.
By Lemma \ref{MGF}, we have
\begin{equation}\label{UpperI}
\frac{1}{n}\sum_{p\in S(k,n)}\log\mathbb{E}\left[e^{\frac{\lambda}{n}\tilde{Y}_{p}^{2}}\right]
\leq\sum_{p\in S(k,n)}\frac{\log(4(n+1))}{n}+4\lambda\left(\frac{1}{p}\right)^{2}\lambda_{1}^{4}.
\end{equation}
Prime number theorem states that
\begin{equation}
\lim_{x\rightarrow\infty}\frac{\pi(x)}{x/\log(x)}=1,
\end{equation}
where $\pi(x)$ denotes the number of primes less than $x$. Therefore, $|k<p\leq n,p\in\mathcal{P}|\leq\frac{2n}{\log n}$
for sufficiently large $n$. Together with \eqref{UpperI}, for sufficiently large $n$, we get
\begin{align}
\frac{1}{n}\log\mathbb{E}\left[e^{\frac{\lambda}{n}\sum_{k<p\leq n,p\in\mathcal{P}}\tilde{Y}_{p}^{2}}\right]
&\leq\frac{2n}{\log n}\frac{\log 4(n+1)}{n}+4\lambda\lambda_{1}^{4}\sum_{k<p\leq n,p\in\mathcal{P}}\frac{1}{p^{2}}\label{UpperII}
\\
&\leq 3+4\lambda\lambda_{1}^{4}\sum_{\ell>k}\frac{1}{\ell^{2}}\nonumber
\\
&\leq 3+4\lambda\lambda_{1}^{4}\frac{1}{k}.\nonumber
\end{align}
Plugging \eqref{UpperII} into \eqref{UpperZero}, we get
\begin{align}
\frac{1}{n}\log\mathbb{P}\left(\sum_{k<p\leq n,p\in\mathcal{P}}\tilde{Y}_{p}^{2}>n^{2}\epsilon\right)
&\leq-\lambda\epsilon+3+\frac{4\lambda\lambda_{1}^{4}}{k}
\\
&=3-\lambda\left(\epsilon-\frac{4\lambda_{1}^{4}}{k}\right).\nonumber
\end{align}
We can choose $\lambda=\frac{1}{4}\log(\epsilon k)-3$ so that $\frac{4\lambda_{1}^{4}}{k}<\frac{\epsilon}{2}$
and it does not violate with our earlier assumption that $\lambda_{1}<\sqrt{2k}$ for large $k$. Hence,
\begin{align}
\frac{1}{n}\log\mathbb{P}\left(\sum_{k<p\leq n,p\in\mathcal{P}}\tilde{Y}_{p}^{2}>\epsilon n^{2}\right)
&\leq 3-\frac{\log(k\epsilon)}{8}\epsilon+3\epsilon
\\
&\leq 4-\frac{\log(k)}{8}\epsilon,\nonumber
\end{align}
which yields the desired result.
\end{proof}

\begin{lemma}\label{K1K2}
Given sufficiently large $k_{1},k_{2}\in\mathbb{N}$, for any sufficiently large $n$, 
\begin{equation}\label{UpperIII}
\frac{1}{n}\log\mathbb{P}\left(\sum_{S(k_{1},k_{2})}Y_{p}^{2}>\epsilon n^{2}\right)
\leq 4\log\log k_{2}+4-\frac{\log(k_{1})}{8}\epsilon.
\end{equation}
\end{lemma}

\begin{proof}
First, observe that $\sum_{p\in S(k_{1},k_{2})}Y_{p}^{2}$ (resp. $\sum_{p\in S(k_{1},k_{2})}\tilde{Y}_{p}^{2}$) only depends on the events
$\{X_{i}\in E_{p_{1},\ldots,p_{\ell}}\}$ (resp. $\{\tilde{X}_{i}\in E_{p_{1},\ldots,p_{\ell}}\}$), where $i,\ell\in\{1,2,\ldots,n\}$ and $\{p_{1},\ldots,p_{\ell}\}\subset S(k_{1},k_{2})$
and
\begin{equation}
E_{p_{1},\ldots,p_{\ell}}:=
\left\{i\in\{1,2,\ldots,n\}|\text{Prime}(i)\cap S(k_{1},k_{2})=\{p_{1},\ldots,p_{\ell}\}\right\},
\end{equation}
where $\text{Prime}(x):=\{q\in\mathcal{P}:\text{$x$ is divisible by $q$}\}$. 
We will show that the following uniform upper bound holds,
\begin{equation}\label{UpperIV}
\frac{\mathbb{P}(X_{1}\in E_{p_{1},\ldots,p_{\ell}})}
{\mathbb{P}(\tilde{X}_{1}\in E_{p_{1},\ldots,p_{\ell}})}
\leq e^{4\log\log k_{2}}.
\end{equation}
Before we proceed, let us show that \eqref{UpperIV} and Theorem \ref{SuperI} implies \eqref{UpperIII}.
Since $X_{i}$'s are independent and $\tilde{X}_{i}$'s are independent,
\begin{equation}
\frac{\mathbb{P}\left(X_{i}\in E_{p_{1}^{i},\ldots,p_{\ell}^{i}},1\leq i\leq n\right)}
{\mathbb{P}\left(\tilde{X}_{i}\in E_{p_{1}^{i},\ldots,p_{\ell}^{i}},1\leq i\leq n\right)}
\leq\left[e^{4\log\log k_{2}}\right]^{n},
\end{equation}
where $\{p_{1}^{i},\ldots,p_{\ell}^{i}\}\subset S(k_{1},k_{2})$ for $1\leq i\leq n$.
Recall that
$\sum_{p\in S(k_{1},k_{2})}Y_{p}^{2}$ (resp. $\sum_{p\in S(k_{1},k_{2})}\tilde{Y}_{p}^{2}$) only depends on the events
$\{X_{i}\in E_{p_{1},\ldots,p_{\ell}}\}$ (resp. $\{\tilde{X}_{i}\in E_{p_{1},\ldots,p_{\ell}}\}$). Therefore,
\begin{align}
\frac{1}{n}\log\mathbb{P}\left(\sum_{p\in S(k_{1},k_{2})}Y_{p}^{2}>n^{2}\epsilon\right)
&\leq 4\log\log k_{2}+\frac{1}{n}\log\mathbb{P}\left(\sum_{p\in S(k_{1},k_{2})}\tilde{Y}_{p}^{2}>n^{2}\epsilon\right)
\\
&\leq 4\log\log k_{2}+4-\frac{\log(k_{1})}{8}\epsilon,\nonumber
\end{align}
where we used Theorem \ref{SuperI} at the last step. Now, let us prove \eqref{UpperIV}.
First, let us give an upper bound for the numerator, that is,
\begin{equation}\label{UpperV}
\mathbb{P}\left(X_{1}\in E_{p_{1},\ldots,p_{\ell}}\right)
=\frac{1}{n}\#|E_{p_{1},\ldots,p_{\ell}}|
\leq\frac{\left[\frac{n}{p_{1}\cdots p_{\ell}}\right]}{n}\leq\frac{1}{p_{1}\cdots p_{\ell}},
\end{equation}
where $[x]$ denotes the largest integer less or equal to $x$ and we used the simple fact that $\frac{[x]}{x}\leq 1$
for any positive $x$. 

As for the lower bound for the denominator, we have
\begin{align}
\mathbb{P}\left(\tilde{X}_{1}\in E_{p_{1},\ldots,p_{\ell}}\right)
&=\prod_{q\in\{p_{1},\ldots,p_{\ell}\}}\frac{1}{q}\prod_{q\in S(k_{1},k_{2})\backslash\{p_{1},\ldots,p_{\ell}\}}
\left(1-\frac{1}{q}\right)\label{UpperVI}
\\
&\geq\prod_{q\in\{p_{1},\ldots,p_{\ell}\}}\frac{1}{q}\prod_{q\in S(k_{1},k_{2})}\left(1-\frac{1}{q}\right)
\nonumber
\\
&\geq\prod_{q\in\{p_{1},\ldots,p_{\ell}\}}\frac{1}{q}e^{-2\sum_{q\in S(k_{1},k_{2})}\frac{1}{q}},
\nonumber
\end{align}
where we used the inequality that $1-x\geq e^{-2x}$ for $x\leq\frac{1}{2}$. Notice that
\begin{equation}
\lim_{n\rightarrow\infty}\left\{-\sum_{q\in S(1,n)}\frac{1}{q}+\log\log n\right\}=M,
\end{equation}
where $M=0.261497\ldots$ is the Meissel-Mertens constant.
Therefore, for sufficiently large $k_{2}$,
\begin{equation}\label{UpperVII}
\sum_{q\in S(k_{1},k_{2})}\frac{1}{q}\leq\sum_{q\in S(1,k_{2})}\frac{1}{q}\leq 2\log\log k_{2}.
\end{equation}
Combining \eqref{UpperV}, \eqref{UpperVI} and \eqref{UpperVII}, we have proved the upper bound in \eqref{UpperIV}.
\end{proof}

\begin{lemma}\label{CRTLemma}
Let $p_{j}$, $1\leq j\leq\ell$, $\ell\in\mathbb{N}$ be the primes such that $S(k_{1},k_{2})=\{p_{1},\ldots,p_{\ell}\}$ and
\begin{equation}
m\prod_{1\leq j\leq\ell}p_{j}\leq n<(m+1)\prod_{1\leq j\leq\ell}p_{j},
\end{equation}
where $m\in\mathbb{N}$. Then, there exists a coupling of vectors of random variables $X_{i}$ and $\tilde{X}_{i}$
for $1\leq i\leq n$, i.e. a measure $\mu$ with marginal distributions the same as $X_{i}$
and $\tilde{X}_{i}$ such that
\begin{equation}\label{UpperVIII}
\mu\left(\sum_{q\in S(k_{1},k_{2})}Y_{q}^{2}
-\sum_{q\in S(k_{1},k_{2})}\tilde{Y}_{q}^{2}\geq n^{2}\epsilon\right)
\leq 2^{n}\left(\frac{1}{m}\right)^{\frac{n\epsilon}{2k_{2}}}.
\end{equation}
\end{lemma}

\begin{proof}
The main ingredient of the proof is the Chinese Remainder Theorem which states that the set of equations
\begin{equation}\label{CRT}
\begin{cases}
x\equiv a_{1} &\text{mod}(p_{1})
\\
\quad\vdots
\\
x\equiv a_{\ell} &\text{mod}(p_{\ell})
\end{cases}
\end{equation}
has a unique solution $1\leq x\leq p_{1}\cdots p_{\ell}$, where $0\leq a_{i}<p_{i}$, $i\in\{1,2,\ldots,\ell\}$.
Hence, for each sequence of $a_{i}$'s, the set of equations in \eqref{CRT} has exactly $m$ solutions
for $1\leq x\leq mp_{1}\cdots p_{\ell}$. We denote these solutions by $R_{i}(a_{1},\ldots,a_{\ell})$ for $i\in\{1,2,\ldots,m\}$.
Given $X_{i}$ uniformly distributed on $\{1,2,\ldots,n\}$, 
we define $\tilde{X}_{i}$ as follows. We generate Bernoulli random variables $c_{j}$ for $1\leq j\leq\ell$,
with parameters $\frac{1}{p_{j}}$ and independent of each other. Now, define
\begin{equation}
\tilde{X}_{i}=
\begin{cases}
p_{1}^{c_{1}}\cdots p_{\ell}^{c_{\ell}} &\text{if $X_{i}>mp_{1}\cdots p_{\ell}$}
\\
p_{1}^{b_{1}}\cdots p_{\ell}^{b_{\ell}} &\text{otherwise}
\end{cases},
\end{equation}
where $b_{j}$ is $1$ if $X_{i}$ is divisible by $p_{j}$ and $0$ otherwise for $1\leq j\leq\ell$.
By the definition, if we condition on $X_{i}>mp_{1}\cdots p_{\ell}$, $\tilde{X}_{i}$ is the multiplication
of $p_{j}^{c_{j}}$ and $c_{j}$'s are independent. Now, conditional on $X_{i}\leq mp_{1}\cdots p_{\ell}$ and
let $\text{Prime}(X_{i})=\{p\in\mathcal{P}:\text{$X_{i}$ is divisible by $p$}\}$. 
Thus, for a vector $\overrightarrow{b}=(b_{j})_{j=1}^{\ell}\in\{0,1\}^{\ell}$, we have
\begin{align}
\Delta &:=\mu\left(\tilde{X}_{i}=\prod_{j=1}^{\ell}p_{j}^{b_{j}}|X_{i}\leq mp_{1}\cdots p_{\ell}\right)
\\
&=\mu\left(\text{Prime}(X_{i})\cap\{p_{1},\ldots,p_{\ell}\}=S(\overrightarrow{b})\right),
\nonumber
\end{align}
where $S(\overrightarrow{b}):=\{p_{j}|b_{j}=1,1\leq j\leq\ell\}$. But that is equivalent to
\begin{align}
\Delta &=\frac{\#\{R_{i}(a_{1},\ldots,a_{\ell})|a_{j}=0 \text{ if and only if } b_{j}=0,1\leq i\leq m\}}{mp_{1}\cdots p_{\ell}}
\\
&=\frac{m\prod_{b_{j}\neq 0}(p_{j}-1)}{mp_{1}\cdots p_{\ell}}\nonumber
\\
&=\prod_{j:b_{j}=0}\frac{1}{p_{j}}\prod_{j:b_{j}\neq 0}\left(1-\frac{1}{p_{j}}\right).\nonumber
\end{align}
Therefore, we get
\begin{equation}
\mu\left(X_{i}=\prod_{j=1}^{\ell}p_{j}^{b_{j}}\right)
=\prod_{j:b_{j}=0}\frac{1}{p_{j}}\prod_{j:b_{j}\neq 0}\left(1-\frac{1}{p_{j}}\right).
\end{equation}
Let us define
\begin{equation}
g(X_{i},\tilde{X}_{i}):=
\begin{cases}
1 &\text{if $\{\text{Prime}(X_{i})\cap S(k_{1},k_{2})\}\neq\{\text{Prime}(\tilde{X}_{i})\cap S(k_{1},k_{2})\}$}
\\
0 &\text{otherwise}
\end{cases}.
\end{equation}
By the definition of the coupling of $\overrightarrow{X}$ and $\overrightarrow{\tilde{X}}$, we have
$\mathbb{P}(g(X_{i},\tilde{X}_{i})=1)\leq\frac{1}{m}$ since the event $g(X_{i},\tilde{X}_{i})=1$ implies
that $X_{i}>mp_{1}\cdots p_{\ell}$ which occurs with probability
\begin{equation}
\frac{n-mp_{1}p_{2}\cdots p_{\ell}}{n}\leq 1-\frac{mp_{1}p_{2}\cdots p_{\ell}}{(m+1)p_{1}\cdots p_{\ell}}
=\frac{1}{m+1}<\frac{1}{m}.
\end{equation}

Now, let us go back to prove the superexponential bound in \eqref{UpperVIII}. 
Observe that
\begin{equation}
f(X_{1},\ldots,X_{n}):=\sum_{q\in S(k_{1},k_{2})}Y_{q}^{2}
=\sum_{q\in S(k_{1},k_{2})}Y_{q}
+\sum_{q\in S(k_{1},k_{2})}\sum_{i\neq j}1_{q|\text{gcd}(X_{i},X_{j})}.
\end{equation}
Hence,
\begin{equation}
\left|\sum_{q\in S(k_{1},k_{2})}Y_{q}^{2}-\tilde{Y}_{q}^{2}\right|\leq\#\{i|g(X_{i},\tilde{X}_{i})=1\}2k_{2}n.
\end{equation}
That is because if we change one of $X_{i}$'s, the function $f(X_{1},\ldots,X_{n})$ changes by at most $k_{2}(n+1)\leq 2k_{2}n$.
Therefore,
\begin{align}
&\mu\left(\sum_{q\in S(k_{1},k_{2})}Y_{q}^{2}-\tilde{Y}_{q}^{2}\geq n^{2}\epsilon\right)
\\
&\leq\mu\left(2k_{2}n\#\{i|g(X_{i},\tilde{X}_{i})=1\}\geq n^{2}\epsilon\right)\nonumber
\\
&=\mu\left(\#\{i|g(X_{i},\tilde{X}_{i})=1\}\geq n\frac{\epsilon}{2k_{2}}\right).\nonumber
\end{align}
Notice that $\#\{i|g(X_{i},\tilde{X}_{i})=1\}=\sum_{i=1}^{n}1_{g(X_{i},\tilde{X}_{i})=1}$ is the sum of i.i.d. indicator functions and $\mu(g(X_{1},\tilde{X}_{1})=1)\leq\frac{1}{m}$.
Hence, by Chebychev's inequality, by choosing $\theta=\log m>0$, we have
\begin{align}
&\mu\left(\#\{i|g(X_{i},\tilde{X}_{i})=1\}\geq n\frac{\epsilon}{2k_{2}}\right)
\\
&\leq\mathbb{E}\left[e^{\theta 1_{g(X_{1},\tilde{X}_{1})=1}}\right]^{n}e^{-\theta n\frac{\epsilon}{2k_{2}}}
\nonumber
\\
&\leq\left(\frac{e^{\theta}}{m}+1\right)^{n}e^{-\theta n\frac{\epsilon}{2k_{2}}}
\nonumber
\\
&\leq2^{n}e^{-(\log m)n\frac{\epsilon}{2k_{2}}}
\nonumber
%\\
%&\leq 2^{n}\left(\frac{1}{m}\right)^{\frac{n\epsilon}{2k_{2}}},
%\nonumber
\end{align}
which yields the desired result.
\end{proof}

\begin{theorem}\label{SuperII}
For any $\epsilon>0$, we have the following superexponential estimates,
\begin{equation}
\limsup_{k\rightarrow\infty}\limsup_{n\rightarrow\infty}\frac{1}{n}
\log\mathbb{P}\left(\sum_{q\in S(k,n)}Y_{q}^{2}>n^{2}\epsilon\right)=-\infty.
\end{equation}
\end{theorem}

\begin{proof}
Let us write
\begin{equation}\label{ThreeTerms}
\sum_{q\in S(k,n)}Y_{q}^{2}
=\sum_{q\in S(k,M_{1})}Y_{q}^{2}
+\sum_{q\in S(M_{1},M_{2})}Y_{q}^{2}
+\sum_{q\in S(M_{2},n)}Y_{q}^{2},
\end{equation}
where $M_{1}:=[\log\log n]^{\frac{120}{\epsilon}}$ and $M_{2}:=[\log n]^{\frac{120}{\epsilon}}$.
By Lemma \ref{K1K2}, for the second and third terms in \eqref{ThreeTerms}, we have
\begin{align}
&\frac{1}{n}\log\mathbb{P}\left(\sum_{q\in S(M_{1},M_{2})}Y_{q}^{2}>\frac{n^{2}\epsilon}{3}\right)\label{UpperIX}
\\
&\leq 4\log\log M_{2}+4-\frac{\log M_{1}}{8}\frac{\epsilon}{3}\nonumber
\\
&=4\log\left(\log\left([\log n]^{\frac{120}{\epsilon}}\right)\right)+4-\frac{\epsilon}{24}\log\left([\log(\log(n))]^{\frac{120}{\epsilon}}\right)
\nonumber
\\
&=4\log\left(\log\left(\frac{120}{\epsilon}\right)+\log\log n\right)+4-5\log\log\log n
\nonumber
\\
&=4\log\log\left(\frac{120}{\epsilon}\right)+4-\log\log\log n,\nonumber
\end{align}
and similarly,
\begin{equation}\label{UpperX}
\frac{1}{n}\log\mathbb{P}\left(\sum_{q\in S(M_{2},n)}Y_{q}^{2}>\frac{n^{2}\epsilon}{3}\right)
\leq-\log(\log n)+4.
\end{equation}
In addition, for the first term in \eqref{ThreeTerms}, by Lemma \ref{CRTLemma}, we get
\begin{equation}\label{UpperXI}
\frac{1}{n}\log\mu\left(\left|\sum_{q\in S(k,M_{1})}Y_{q}^{2}-\sum_{q\in S(k,M_{1})}\tilde{Y}_{q}^{2}\right|>\frac{n^{2}\epsilon}{6}\right)
\leq\log 2-\frac{\epsilon}{12M_{1}}\log M_{0},
\end{equation}
where
\begin{align}
M_{0}&:=\frac{n}{\prod_{q\in S(k,M_{1})}q}\label{UpperXII}
\\
&\geq\frac{n}{M_{1}^{M_{1}}}\nonumber
\\
&=\exp\left\{\log(n)-\frac{120}{\epsilon}(\log\log n)^{\frac{120}{\epsilon}}\log\log\log n\right\}.\nonumber
\end{align}
By Theorem \ref{SuperI}, 
\begin{equation}\label{UpperXIII}
\frac{1}{n}\log\mathbb{P}\left(\sum_{p\in S(k,M_{1})}\tilde{Y}_{p}^{2}\geq\frac{n^{2}\epsilon}{6}\right)
\leq-\frac{\epsilon}{48}\log(k)+4.
\end{equation}
Combining \eqref{UpperIX}, \eqref{UpperX}, \eqref{UpperXI}, \eqref{UpperXII} and \eqref{UpperXIII}, we get the desired result.
\end{proof}

Finally, we are ready to prove Theorem \ref{LDPThm}.

\begin{proof}[Proof of Theorem \ref{LDPThm}]
We let $U_{i}$, for $1\leq i\leq n$, be i.i.d. random variables chosen from measure $\nu$ as in \eqref{my_nu}. In addition, we define $U_{i}^{k}$, for $k\in \mathbb{N}$, as the restriction of $U_{i}$  to its first $k$ digits, i.e. 
\begin{equation}
\chi_{j}(U_{i}^{k}) =
\begin{cases}
\chi_{j}(U_{i}) &\text{if $j\leq k$}
\\
0 &\text{if $j>k$}
\end{cases},
\end{equation}
where $\chi$ is defined in \eqref{chi}.

Let $L_{n}$, $L_{n}^{k}$ be the empirical measures of $U_{i}$, $U_{i}^{k}$, i.e.
\begin{equation}
L_{n}(x):=\frac{1}{n}\sum_{i=1}^{n}\delta_{U_{i}}(x),
\end{equation}
and
\begin{equation}
L_{n}^{k}(x):=\frac{1}{n}\sum_{i=1}^{n}\delta_{U_{i}^{k}}(x).
\end{equation}

In large deviations theory, Sanov's theorem (see e.g. Dembo and Zeitouni \cite{Dembo}) says that,  
for a sequence of i.i.d. random variables $X_{1},X_{2},\ldots,X_{n}$ taking values in a Polish space $\mathbb{X}$
with common distribution $\alpha\in\mathcal{M}(\mathbb{X})$, the space of probability measures on $\mathbb{X}$
equipped with weak topology, the probability measures $\mathbb{P}(\frac{1}{n}\sum_{i=1}^{n}\delta_{X_{i}}\in\cdot)$ 
induced by the empirical measures $\frac{1}{n}\sum_{i=1}^{n}\delta_{X_{i}}$ satisfy
a large deviation principle with rate function $I(\beta)$ given by
\begin{equation}
I(\beta)=\int_{\mathbb{X}}\frac{d\beta}{d\alpha}\log\frac{d\beta}{d\alpha}\alpha(dx),
\end{equation}
if $\beta\ll\alpha$ and $\frac{d\beta}{d\alpha}|\log\frac{d\beta}{d\alpha}|\in L^{1}(\alpha)$ and $I(\beta)=+\infty$ otherwise.

Therefore, by Sanov's theorem,
$\mathbb{P}(L_{n}\in\cdot)$ satisfies a large deviation
principle on $\mathcal{M}[0,1]$, the space of probability measures on $[0,1]$, equipped with the weak topology
and the rate function
\begin{equation}
I(\mu)=
\begin{cases}
\int_{[0,1]}\log\left(\frac{d\mu}{d\nu}\right)d\mu &\text{if $\mu\ll\nu$ and $|\log\frac{d\mu}{d\nu}|\in L^{1}(\mu)$}
\\
+\infty &\text{otherwise}
\end{cases}.
\end{equation}

We define $f_{k}:[0,1]^{2}\rightarrow\{0,1\}$, for $k\in\mathbb{N}$, and redefine $f$ from \eqref{my_rate} as follows
\begin{equation}
f_{k}(x_{1},x_{2})=1-\max_{1\leq i\leq k}\chi_{i}(x_{1})\chi_{i}(x_{2}),
\quad
\text{and}
\quad
f(x_{1},x_{2}):=1-\max_{i\in\mathbb{N}}\chi_{i}(x_{1})\chi_{i}(x_{2}).
\end{equation}
In other words, $f$ is $1$ if $x_{1}$ and $x_{2}$ do not share a common $1$ at the same place in their binary expansions
and $f$ is $0$ otherwise. Similar interpretation holds for $f_{k}$. Clearly, $f_{k}\geq f$ and
$\lim_{k\rightarrow\infty}f_{k}(x_{1},x_{2})=f(x_{1},x_{2})$. Again, let $\nu$ be the probability measure on $[0,1]$ such
that for a random variable $x$ with measure $\nu$, $\chi_{i}(x)$ are i.i.d. Bernoulli random variables
with parameters $\frac{1}{p_{i}}$, where $p_{i}$ is the $i$th smallest prime number.

Let $\alpha^{k}:=\{\alpha\in[0,1]|\chi_{i}(\alpha)=0\text{ for }i>k\}$ be the set
of numbers on $[0,1]$ with $k$-digit binary expansion. We define
\begin{equation}
A_{\alpha}:=\{x\in[0,1]|\chi_{i}(\alpha)=\chi_{i}(x), 1\leq i\leq k\}.
\end{equation}
Let $F_{k}(\mu):=\iint_{[0,1]^{2}}f_{k}(x_{1},x_{2})d\mu(x_{1})d\mu(x_{2})$ 
and $F(\mu):=\iint_{[0,1]^{2}}f(x_{1},x_{2})d\mu(x_{1})d\mu(x_{2})$. 
We have
\begin{equation}
F_{k}(\mu)=\sum_{\alpha,\beta\in\alpha^{k}}f_{k}(\alpha,\beta)\mu(A_{\alpha})\mu(A_{\beta}).
\end{equation}
Hence the map $\mu\mapsto F_{k}(\mu)$ is continuous, i.e. for $\mu_{n}\rightarrow\mu$ in the weak topology,
$F_{k}(\mu_{n})\rightarrow F_{k}(\mu)$. 
In large deviations theory, the contraction principle (see e.g. Dembo and Zeitouni \cite{Dembo}) says that
if $\mathbb{P}_{n}$ satisfies a large deviation principle on a Polish space $\mathbb{X}$ with rate function $I(\cdot)$
and $F$ is a continuous mapping from $\mathbb{X}$ to another Polish space $\mathbb{Y}$, then $\mathbb{P}_{n}F^{-1}$
satisfies a large deviation principle on $\mathbb{Y}$ with a rate function $J(\cdot)$ given by $J(y)=\inf_{x:F(x)=y}I(x)$.

Therefore, by the contraction principle, $\mathbb{P}(L_{n}\circ F_{k}^{-1}\in\cdot)$ 
satisfies a large deviation principle with good rate function
\begin{equation}
I^{(k)}(x)=\inf_{\iint_{[0,1]^{2}}f_{k}(x_{1},x_{2})d\mu(x_{1})d\mu(x_{2})=x}\int_{[0,1]}\log\left(\frac{d\mu}{d\nu}\right)d\mu.
\end{equation}
Moreover, in Theorem \ref{SuperI}, we proved that
\begin{equation}
\limsup_{k\rightarrow\infty}\limsup_{n\rightarrow\infty}
\frac{1}{n}\log\mathbb{P}\left(\iint_{[0,1]^{2}}(f_{k}-f)dL_{n}(x)dL_{n}(y)\geq\delta\right)=-\infty,
\end{equation}
for any $\delta>0$. In other words, 
the family $\{L_{n}\circ F_{k}^{-1}\}$ are exponentially good approximation of $\{L_{n}\circ F^{-1}\}$, 
see Definition 4.2.14 in Dembo and Zeitouni \cite{Dembo}.
Now, by Theorem 4.2.16 in Dembo and Zeitouni \cite{Dembo}, $\mathbb{P}(L_{n}\circ F^{-1}\in\cdot)$ satisfies a weak large deviation principle (for the definition of weak large deviation principle, we refer to page 7 of Dembo and Zeitouni \cite{Dembo}) 
with the rate function
\begin{equation}
I_{1}(x)=\sup_{\delta>0}\liminf_{k\rightarrow\infty}\inf_{|w-x|<\delta}I^{(k)}(w).
\end{equation}
Since the interval $[0,1]$ is compact, $\mathbb{P}(L_{n}\circ F^{-1}\in\cdot)$ satisfies the full large deviation principle with
good rate function $I_{1}(x)$ as above and it is easy to check that
\begin{equation}
I_{1}(x)=\inf_{\iint_{[0,1]^{2}}f(x_{1},x_{2})\mu(dx_{1})\mu(dx_{2})=x}\int_{[0,1]}\log\left(\frac{d\mu}{d\nu}\right)d\mu.
\end{equation}

For any $p\in\mathcal{P}$, let us recall that $Y_{p}=\sum_{i=1}^{n}1_{p|X_{i}}$, $\tilde{Y}_{p}=\sum_{i=1}^{n}1_{p|\tilde{X}_{i}}$, 
and for any $k_{1},k_{2}\in\mathbb{N}$, 
$S(k_{1},k_{2})=\{p\in\mathcal{P}:k_{1}<p\leq k_{2}\}$.

By Theorem \ref{SuperI}, we have
\begin{align}
&\limsup_{k\rightarrow\infty}\limsup_{n\rightarrow\infty}\frac{1}{n}\log
\mathbb{P}\left(\frac{1}{n^{2}}
\sum_{1\leq i,j\leq n}1_{\text{gcd}(\tilde{X}_{i}^{k},\tilde{X}_{j}^{k})=1}-1_{\text{gcd}(\tilde{X}_{i},\tilde{X}_{j})=1}\geq\epsilon\right)
\\
&=\limsup_{k\rightarrow\infty}\limsup_{n\rightarrow\infty}\frac{1}{n}\log
\mathbb{P}\left(\frac{1}{n^{2}}
\sum_{1\leq i,j\leq n}\sum_{p\in S(k,n)}1_{p|\tilde{X}_{i},p|\tilde{X}_{j}}\geq\epsilon\right)
\nonumber
\\
&\leq\limsup_{k\rightarrow\infty}\limsup_{n\rightarrow\infty}\frac{1}{n}\log
\mathbb{P}\left(\frac{1}{n^{2}}
\sum_{p\in S(k,n)}\tilde{Y}_{p}^{2}\geq\epsilon\right)
\nonumber
\\
&=-\infty.\nonumber
\end{align}
Next, notice that the difference between $\mathbb{P}(\frac{1}{n^{2}}\sum_{1\leq i,j\leq n}1_{\text{gcd}(X_{i}^{k},X_{j}^{k})=1}\in\cdot)$
and $\mathbb{P}(\frac{1}{n^{2}}\sum_{1\leq i,j\leq n}1_{\text{gcd}(\tilde{X}_{i}^{k},\tilde{X}_{j}^{k})=1}\in\cdot)$ is superexponentially small by Lemma \ref{CRTLemma}.
Finally, by Theorem \ref{SuperII},
\begin{equation}
\limsup_{k\rightarrow\infty}\limsup_{n\rightarrow\infty}\frac{1}{n}\log
\mathbb{P}\left(\frac{1}{n^{2}}
\sum_{1\leq i,j\leq n}1_{\text{gcd}(X_{i}^{k},X_{j}^{k})=1}-1_{\text{gcd}(X_{i},X_{j})=1}\geq\epsilon\right)
=-\infty.
\end{equation}
This implies that 
\begin{align}
-\inf_{x\in A^{o}}I_{1}(x)&\leq\liminf_{n\rightarrow\infty}\frac{1}{n}\log
\mathbb{P}\left(\frac{1}{n^{2}}\sum_{1\leq i,j\leq n}1_{\text{gcd}(X_{i},X_{j})=1}\in A\right)
\\
&\leq\limsup_{n\rightarrow\infty}\frac{1}{n}\log
\mathbb{P}\left(\frac{1}{n^{2}}\sum_{1\leq i,j\leq n}1_{\text{gcd}(X_{i},X_{j})=1}\in A\right)
\leq-\inf_{x\in\overline{A}}I_{1}(x).
\nonumber
\end{align}
\end{proof}

\section{Proofs of Central Limit Theorem}\label{CLTProofs}

\begin{proof}[Proof of Theorem \ref{CLTThm}]
Here, we prove our result for $\ell=1$. The proof for $\ell>1$ is the same that is skipped.

Instead of summing over $1\leq i,j\leq n$, we only need to consider $1\leq i\neq j\leq n$.
The reason is because if $i=j$, then $\text{gcd}(X_{i},X_{i})=1$ if and only if $X_{i}=1$
which occurs with probability $\frac{1}{n}$ and therefore $\frac{1}{2n^{3/2}}\sum_{i=1}^{n}1_{\text{gcd}(X_{i},X_{i})=1}$
is negligible in the limit as $n\rightarrow\infty$. Moreover, 
\begin{equation}
\sum_{1\leq i\neq j\leq n}1_{\text{gcd}(X_{i},X_{j})=1}
=2\sum_{1\leq i<j\leq n}1_{\text{gcd}(X_{i},X_{j})=1},
\end{equation}
and we can therefore concentrate on $1\leq i<j\leq n$.

Let us define $a_{ij}=1_{\text{gcd}(X_{i},X_{j})=1}$ for $1\leq i<j\leq n$.
$a_{ij}$ have the same distribution and let $\alpha_{n}$ be the mean of $a_{12}$.
Then, we have
\begin{equation}
\alpha_{n}=\mathbb{E}[a_{12}]=\mathbb{P}(\text{gcd}(X_{1},X_{2})=1)
\rightarrow\prod_{p\in\mathcal{P}}\left(1-\frac{1}{p^{2}}\right),
\end{equation}
as $n\rightarrow\infty$.
Define $\tilde{a}_{ij}:=a_{ij}-\alpha_{n}$ and $W=\sum_{(i,j)\in I}\tilde{a}_{ij}$,
where the sum is taken over the set $I$ that is all the pairs of $i,j\in[n]$ and $i<j$.
Therefore,
\begin{equation}
\sigma_{n}^{2}:=\text{Var}(W)=\mathbb{E}\left[\left(\sum_{(i,j)\in I}\tilde{a}_{ij}\right)^{2}\right]
=\sum_{(i,j)}\sum_{(\ell,k)}\mathbb{E}[\tilde{a}_{ij}\tilde{a}_{k\ell}].
\end{equation}
Note that if the intersection of $\{i,j\}$ and $\{k,\ell\}$ is empty, then
$\tilde{a}_{ij}$ and $\tilde{a}_{k\ell}$ are independent and
$\mathbb{E}[\tilde{a}_{ij}\tilde{a}_{k\ell}]=\mathbb{E}[\tilde{a}_{ij}]
\mathbb{E}[\tilde{a}_{k\ell}]=0$. The remaining two cases are either
$\{i,j\}=\{k,\ell\}$ or $|\{i,j\}\cap\{k,\ell\}|=1$. 
For the former, we have
\begin{equation}
\mathbb{E}[\tilde{a}_{ij}\tilde{a}_{ij}]=\mathbb{E}[a_{ij}^{2}]-\alpha_{n}^{2}
=\mathbb{E}[a_{ij}]-\alpha_{n}^{2}=\alpha_{n}-\alpha_{n}^{2}.
\end{equation}
For the latter, assuming without loss of generality that $i=k$ and $j\neq\ell$,
we get
\begin{align}
\mathbb{E}\left[\tilde{a}_{ij}\tilde{a}_{k\ell}\right]
&=\mathbb{E}\left[a_{ij}a_{i\ell}\right]-\alpha_{n}^{2}
\\
&=\mathbb{P}\left(\text{gcd}(X_{i},X_{j})=\text{gcd}(X_{i},X_{\ell})=1\right)-\alpha_{n}^{2}
\nonumber
\\
&=\beta_{n}-\alpha_{n}^{2},\nonumber
\end{align}
where $\beta_{n}:=\mathbb{P}\left(\text{gcd}(X_{i},X_{j})=\text{gcd}(X_{i},X_{\ell})=1\right)$.

Let $p$ be a prime number and $\tilde{X}_{1}$, $\tilde{X}_{2}$, $\tilde{X}_{3}$ be defined as before, that is, 
three i.i.d. integer valued random
variables so that $\tilde{X}_{1}$ is divisible by $p\leq n$ with probability $\frac{1}{p}$. Then, by inclusion-exclusion principle,
it is easy to see that 
\begin{equation}
\lim_{n\rightarrow\infty}\beta_{n}
=\lim_{n\rightarrow\infty}
\mathbb{P}(\text{gcd}(\tilde{X}_{1},\tilde{X}_{2})=\text{gcd}(\tilde{X}_{1},\tilde{X}_{3})=1)
=\prod_{p\in\mathcal{P}}\left(1-\frac{2}{p^{2}}+\frac{1}{p^{3}}\right).
\end{equation}

We have $\binom{n}{2}$ pairs that $\{i,j\}=\{k,\ell\}$
and $3\times 2\times\binom{n}{3}$ pairs that $|\{i,j\}\cap\{k,\ell\}|=1$.
(We pick three numbers from $1$ to $n$. Then we pick one of them to be
duplicated, say $i$. Finally, we have two pairs as $(i,j)(i,k)$
and $(i,k)(i,j)$). Thus
\begin{equation}\label{sigmasquare}
\sigma_{n}^{2}=\binom{n}{2}(\alpha_{n}-\alpha_{n}^{2})+3\cdot 2\cdot\binom{n}{3}(\beta_{n}-\alpha_{n}^{2}),
\end{equation}
and we have
\begin{equation}
\frac{\sigma_{n}^{2}}{n^{3}}\rightarrow\prod_{p\in\mathcal{P}}
\left(1-\frac{2}{p^{2}}+\frac{1}{p^{3}}\right)-\frac{36}{\pi^{4}},
\end{equation}
as $n\rightarrow\infty$.

Now, our goal is to use the general theorem for random dependency graphs
to prove that $W=\frac{1}{\sigma_{n}}\sum_{(i,j)\in I}\tilde{a}_{ij}$
converges to a standard normal random variable.

We have a collection of dependent random variables $(\tilde{a}_{ij})_{(i,j)\in I}$.
We say $\tilde{a}_{ij}$ and $\tilde{a}_{k\ell}$ are neighbors if they
are dependent, i.e. $\{i,j\}\cap\{k,\ell\}\neq\emptyset$.

Let $N(i,j)=\{\text{neighbors of }(i,j)\}\cup\{(i,j)\}$.
Hence, $N(i,j)$ has $D=2n-5$ elements. In addition, let $Z$
be a standard normal random variable.

By Theorem 3.6. of Ross \cite{Ross}, we have
\begin{equation}
d_{W}(W,Z)
\leq\frac{D^{2}}{\sigma_{n}^{3}}\sum_{(i,j)\in I}\mathbb{E}|\tilde{a}_{ij}|^{3}
+\frac{\sqrt{28}}{\sqrt{\pi}}\frac{D^{3/2}}{\sigma_{n}^{2}}
\sqrt{\sum_{(i,j)\in I}\mathbb{E}|\tilde{a}_{ij}|^{4}}.
\end{equation}
Note that $|\tilde{a}_{ij}|$ is bounded by $1$. Thus, using \eqref{sigmasquare}, we have
\begin{align}
d_{W}(W,Z)&\leq\frac{D^{2}}{\sigma_{n}^{3}}\binom{n}{2}
+\frac{\sqrt{28}}{\sqrt{\pi}}\frac{D^{3/2}}{\sigma_{n}^{2}}\sqrt{\binom{n}{2}}
\\
&\leq\frac{(2n)^{2}\cdot n^{2}}{\sigma_{n}^{3}}+5\frac{(2n)^{3/2}n}{\sigma_{n}^{2}}
\nonumber
\\
&\leq\frac{C}{n^{1/2}}.\nonumber
\end{align}
where $C$ is a universal constant.
For the general $\ell\in\mathbb{N}$, $d_{W}(W,Z)\leq\frac{C_{\ell}}{n^{1/2}}$.
Note that $\lim{n\rightarrow\infty}\frac{\sigma_{n}^{2}}{n^{3}}=O(\ell^{-3})$
as $\ell\rightarrow\infty$ and also note that
$\mathbb{E}|\tilde{a}_{ij}|^{3}\leq\mathbb{E}|\tilde{a}_{ij}|=O(\ell^{-2})$
as $\ell\rightarrow\infty$. 
Therefore, $C_{\ell}=O(\ell^{5/2})$ as $\ell\rightarrow\infty$.
\end{proof}

\section*{Acknowledgements}

The authors would like to thank the anonymous referee for a very careful reading of the manuscript
and helpful suggestions which greatly improve the paper. The authors also thank the Editor Professor Ben Green
for helpful comments.
The authors are also very grateful to Professor S. R. S. Varadhan for helpful discussions and generous suggestions.


\begin{thebibliography}{7}

\bibitem{CesaroIII}
Ces\`{a}ro, E. (1881).
D\'{e}monstration \'{e}l\'{e}mentaire et g\'{e}n\'{e}ralisation de quelques th\'{e}r\`{e}mes de M. Berger.
\textit{Mathesis}. \textbf{1}, 99-102.

\bibitem{CesaroI}
Ces\`{a}ro, E. (1885).
\'{E}tude moyenne du plus grand commun diviseur de deux nombres.
\textit{Annali di Matematica Pura ed Applicata}.
\textbf{13}, 235-250.

\bibitem{CesaroII}
Ces\`{a}ro, E. (1885).
Sur le plus grand commun diviseur de plusieurs nombres.
\textit{Annali di Matematica Pura ed Applicata}.
\textbf{13}, 291-294.

\bibitem{Cohen}
Cohen, E. (1960).
Arithmetical functions of a greatest common divisor. I.
\textit{Proc. Amer. Math. Soc.}
\textbf{11}, 164-171.

\bibitem{Dembo} 
Dembo, A. and O. Zeitouni. 
\emph{Large Deviations Techniques and Applications}, 2nd Edition, Springer, New York, 1998.

\bibitem{Diaconis}
Diaconis, P. and P. Erd\H{o}s.
On the distribution of the greatest common divisor. 
Technical Report No. 12. Stanford University, 1977. Reprinted in
\textit{A Festschrift for Herman Rubin}, 56-61.
Lecture Notes, Monograph Series, Vol. 45, Institute of Mathematical Statistics, 2004.

\bibitem{Elliott}
Elliott, P. D. T. A.
\textit{Probabilistic Number Theory}, Volume I and II, Springer-Verlag, New York, 1980. 

\bibitem{FernandezI}
Fern\'{a}ndez, J. L. and P. Fern\'{a}ndez. (2015).
Asymptotic normality and greatest common divisors.
\textit{International Journal of Number Theory}.
\textbf{11}, 89.

\bibitem{FernandezII}
Fern\'{a}ndez, J. L. and P. Fern\'{a}ndez. (2013).
On the probability distribution of gcd and lcm of $r$-tuples of integers.
\textit{arXiv:1305.0536}.

\bibitem{Hardy}
Hardy, G. H. and  E. M. Wright.
\textit{An Introduction to the Theory of Numbers}, 6th Edition, Oxford University Press, 2008.

\bibitem{Ross}
Ross, N. (2011).
Fundamentals of Stein's method.
\textit{Probability Surveys}.
\textbf{8}, 210-293.

\bibitem{Tenenbaum}
Tenenbaum, G.
\textit{Introduction to Analytic and Probabilistic Number Theory},
Cambridge Studies in Advanced Mathematics, Cambridge University Press, 1995.

\bibitem{VaradhanII} Varadhan, S. R. S. 
\textit{Large Deviations and Applications}, SIAM, Philadelphia, 1984.

\end{thebibliography}
\end{document}